\documentclass[11pt]{amsart}
\usepackage[UKenglish]{babel}
\usepackage[utf8]{inputenc} 
\usepackage{amsthm} 
\usepackage{amsfonts} 
\usepackage{mathtools}
\usepackage{enumitem} \setlist[enumerate]{label={\upshape(\arabic*)}}
\usepackage{amssymb}
\usepackage[numbers]{natbib}
\usepackage{url}
\usepackage{mathrsfs}
\usepackage{geometry}
\usepackage{bbm}
\usepackage{tikz-cd}
\usepackage{tikz}
\usepackage{color}
\usepackage{framed}
\usepackage{hyperref}
\usepackage{cleveref}
\usepackage[disable]{todonotes}
\hypersetup{
    colorlinks=true, 
    linkcolor=blue, 
    urlcolor=red, 
    citecolor=[rgb]{0,0.7,0},
    linktoc=all 
}

 \theoremstyle{definition}
 \newtheorem{definition}{Definition}[section]
 \newtheorem{proposition}[definition]{Proposition}
 \newtheorem{theorem}[definition]{Theorem}
 \newtheorem{lemma}[definition]{Lemma}
 \newtheorem{corollary}[definition]{Corollary}
 \newtheorem{rmk}[definition]{Remark}
 \newtheorem{ex}[definition]{Example}

\newcommand{\R}{\mathbb{R}}

\newcommand{\Z}{\mathbb{Z}}

\DeclareMathOperator{\conv}{conv}

\def\1{\mathbf{1}}

\title[Column number bounds on polytopal TU-matrices and unimodular polytopes]{Unimodular polytopes and column number bounds on polytopal totally unimodular matrices via Seymour's decomposition theorem}
\author{Benjamin Nill}

\address{Faculty of Mathematics, Otto-von-Guericke-Universit\"at Magdeburg, Universit\"atsplatz 2, 39106 Magdeburg, Germany. Email: benjamin.nill@ovgu.de}

\begin{document}

\setlength{\parindent}{0pt}
\begin{abstract}
We prove a sharp upper bound on the number of distinct columns of a totally unimodular matrix with column sums $1$ improving upon Heller's classical bound. The proof uses Seymour's decomposition theorem. Such matrices are closely related to unimodular polytopes: lattice polytopes where the vertices of every full-dimensional subsimplex form an affine lattice basis. This is an interesting subclass of 0/1-polytopes and contains for instance edge polytopes of bipartite graphs. Our main result on totally unimodular matrices implies a sharp upper bound on the number of vertices of unimodular polytopes. 
\end{abstract}

\maketitle

\section{The column number bound on polytopal TU-matrices}

Totally unimodular matrices are objects of fundamental importance in optimization and matroid theory \cite{Schrbook}. 

\begin{definition}
    An $m \times n$-matrix $A$ with entries in $\Z$ is called {\em totally unimodular} or {\em TU} if every minor is in $\{-1,0,1\}$.
\end{definition}

Note that particularly every entry of a TU-matrix is in $\{-1,0,-1\}$. Heller \cite{Heller} showed that any TU-matrix with $m$ rows such that all columns are pairwise distinct has at most $m^2+m+1$ columns. This is a sharp bound: take the complete graph on $m$ nodes and replace each edge with two arcs for both orientations. Now, extend the $m \times m(m-1)$-incidence matrix of this directed graph by the $m \times m$-identity matrix, the $m \times m$-identity matrix times $-1$, and the zero column to get a TU-matrix with $m$ rows and $m^2+m+1$ columns.

There has been lots of recent activity to study extensions of Heller's result: to find sharp upper bounds on the number of distinct columns of integer matrices satisfying certain constraints, most prominently, bounded determinants. This is also called the column number problem \cite[Problem~1.1]{Paat}. We refer to \Cref{bounded} for more references on current research.

Our contribution is to solve the column number problem for TU-matrices where all columns lie in a common affine hyperplane not containing the origin -- we call such TU-matrices {\em polytopal}. This is motivated by the study of unimodular polytopes which will be discussed in detail in the next section. 

In order to make the similarity of our column number bound to Heller's bound apparent let us rewrite it in the following equivalent form. Here, $(\Z^{m})^*$ denotes the set of integral functionals on $\Z^m$.

\begin{theorem}[Heller's bound]
Let $M$ be a TU-matrix with $m$ rows and $n$ columns such that all columns are pairwise distinct. If there exists $f \in (\Z^{m})^*$ such that $f$ evaluates to a positive number on every column of $M$ (e.g., all column sums are positive), then $n$ is bounded from above by $\frac{(m+1)m}{2}$.
\end{theorem}

To see the equivalence to the original formulation of Heller, for the one direction consider the extended matrix consisting of $M$, $-M$, and the zero column, while for the reverse direction note that one can always separate a finite number of non-zero vectors by an integral functional into two halfspaces.

\smallskip

Here is the definition of polytopal integer matrices.

\begin{definition}
An $m \times n$-matrix integer matrix $M$ is called {\em polytopal} if there exists $f \in (\Z^{m})^*$ such that $f$ evaluates to $1$ on every column of $M$.
\end{definition}

Note that for TU-matrices this is equivalent\footnote{For the `if'-direction let us assume that the first $r$ columns $v_1, \ldots, v_r$ of $M$ are linearly independent and span the column span of $M$. By our assumption, all other columns of $M$ are contained in the affine hull of $v_1, \ldots, v_r$. By the TU-property we can extend $v_1, \ldots, v_r$ to a lattice basis $B$ of $\Z^m$. Clearly, one can choose $f \in (\Z^{m})^*$ evaluating to $1$ on $B$, and thus on all columns of $M$.} to all columns lying in a common affine hyperplane that does not contain the origin. An important special case is that all column sums are $1$. Here is our main result.

\begin{theorem}\footnote{After the publication of this article Jon Lee pointed us to Corollary~4.3 in his paper \cite{JonLee}, which directly implies Theorem~\ref{theorem:tu}.}
    Let $M$ be a TU-matrix with $m$ rows and $n$ columns such that all columns are pairwise distinct. If $M$ is polytopal, then $n$ is bounded from above by $10$ if $m=5$ and by $\lfloor \frac{(m+1)^2}{4} \rfloor$ if $m \not=5$.\label{theorem:tu}
\end{theorem}

We remark that this sharpens in this polytopal situation Heller's bound (roughly) by a factor of $2$. 

\begin{ex} The bound in \Cref{theorem:tu} is sharp for all $m$ as the following examples show. For $m=5$ one may take the TU polytopal $5 \times 10$-matrix
\[\begin{pmatrix}1&0&0&0&0&1&0&0&1&-1\\0&1&0&0&0&-1&1&0&0&1\\0&0&1&0&0&1&-1&1&0&0\\0&0&0&1&0&0&1&-1&1&0\\0&0&0&0&1&0&0&1&-1&1\end{pmatrix}\]
For $m \not=5$ and $m$ odd, take the incidence matrix (with entries in $0$ and $1$) of a complete bipartite graph $K_{(m+1)/2,(m+1)/2}$ with parts $A,B$ and remove a row corresponding to a vertex in $B$. This yields a polytopal $m \times (m+1)^2/4$-matrix as summing for each column over the $(m+1)/2$ entries corresponding to the vertices of $A$ equals $1$.
Similarly, for $m \not=5$ and $m$ even, the bound is attained by the incidence matrix of the complete bipartite graph $K_{m/2,m/2+1}$ with one row removed. Note that $(m/2)(m/2+1) = (m+1)^2/4-\frac{1}{4} = \lfloor \frac{(m+1)^2}{4} \rfloor$.
\end{ex}

As the exceptional case for $m=5$ already hints at, the proof of \Cref{theorem:tu} relies on the celebrated decomposition theorem of Seymour. We note that Seymour's decomposition theorem has also been used in the paper by Bixby and Cunningham \cite{bixby} to get a general result for matroids from which they derived an alternative proof of Heller's bound on TU-matrices. They also described a beautiful direct combinatorial argument of Heller's theorem via Sauer's lemma which is discussed in Schrijver's book \cite[section 21.3]{Schrbook} and has been recently used to get bounds on more general TU-matrices with bounded minors \cite{averkov2024maximal}. However, in our setting it is an open problem whether such a direct argument is possible.\footnote{The elegant arguments in \cite{JonLee} also rely on Seymour's decomposition results.}

Our motivation for the study of polytopal TU-matrices comes from their direct relation to unimodular polytopes, see \Cref{translation} below. In the next section we will discuss this interesting class of lattice polytopes and use \Cref{theorem:tu} to deduce a sharp upper bound on their number of vertices in \Cref{main}. The proof of \Cref{theorem:tu} will occupy Section~\ref{sec:proof}. 

\begin{rmk}            
    We believe\footnote{Also this generalization is implied by Corollary~4.3 in \cite{JonLee}.} that our bound in \Cref{theorem:tu} also holds for TU-matrices such that all columns are pairwise distinct and have positive {\em odd} column sums. In fact, the proof of \Cref{theorem:tu} could be immediately extended to this case, the only missing part is the computationally involved task to verify that it holds in the case of at most $6$ rows (which is true for column sums equal to $1$ by \Cref{fewrows}).
\end{rmk}

\subsection*{Acknowledgment}

We thank Gennadiy Averkov, Kerstin Bennecke, Volker Kaibel and Matthias Schymura for their interest and remarks. We are extremely grateful to the anonymous referees for their very thorough reading and several very useful remarks such as pointing out the known column number bound in the cographical case. The author is funded by the Deutsche Forschungsgemeinschaft (DFG, German Research Foundation) – 539867500 as part of the research priority program Combinatorial Synergies. We thank Jon Lee for letting us know about his paper \cite{JonLee}.

\section{Applying the column number bound to unimodular polytopes}

\label{sec:unimodular}

As is well-known, many interesting families of lattice polytopes can be defined via totally unimodular matrices \cite{Schrbook, Christian}, for instance from incidence matrices of directed graphs. In this section, we will apply \Cref{theorem:tu} to the class of unimodular polytopes. 

\subsection{Definition and relation to polytopal TU-matrices}

Let $P \subset \R^d$ be a {\em lattice polytope}, i.e., $P$ is the convex hull of points in the lattice $\Z^d$. Two lattice polytopes are {\em isomorphic} (we use the symbol $\cong$) if there is an affine-linear automorphism of $\Z^d$ mapping the vertices bijectively to each other.

A lattice {\em simplex} $S$ of dimension $d$ is called {\em unimodular} if its vertices form an affine lattice basis of $\Z^d$. Equivalently, $S \cong \Delta_d := \conv(0,e_1, \ldots, e_d)$. Note that a unimodular simplex contains only its vertices as lattice points. The definition of unimodular simplices can be extended to lattice polytopes (see \cite[p.~297]{Triangulationsbook} and \cite[Section~4.3]{Binomial}).

\begin{definition}
    $P$ is called \emph{unimodular polytope} if every full-dimensional simplex whose vertex set is a subset of the vertices of $P$ is unimodular.
\end{definition}

\begin{ex} Consider an undirected bipartite graph $G$ with parts $A$ and $B$. Then the {\em edge polytope} of $G$ is defined as the convex hull of the vectors $e_i+e_j$ where $i \in A$ and $j \in B$ form an edge of $G$. 
For the complete bipartite graph $K_{a,b}$, we have $P(K_{a,b}) \cong \Delta_{a-1} \times \Delta_{b-1}$. Hence, by \cite[Prop.~6.2.11]{Triangulationsbook} any edge polytope of a bipartite graph is unimodular.\label{ex:ex}
\end{ex}

\begin{rmk}
    Instead of `unimodular polytope' also the term `totally unimodular polytope' is used in the literature, e.g., \cite[Proposition~9.3.20]{Triangulationsbook}. The reader should be cautious that there is also the notion of a {\em unimodular point configuration}, which means that every linearly independent subfamily of full rank has determinant $1$ or $-1$. The convex hull of a unimodular point configuration may not be a unimodular polytope (e.g., for $[-1,1])$), while the vertex set of a unimodular polytope does not have to be a unimodular point configuration (e.g., for $[1,2]$). 
    
    Unimodular point configurations and their convex hulls turn up naturally for interesting classes of lattice polytopes \cite{Binomial, Nef}. The convex hull of columns of the incidence matrix of a directed graph is also known as an {\em arc polytope}, see \cite[Section 3.2.1]{Triangulationsbook} and references therein. Recently, motivated from their connection to (oriented) regular matroids there has been a lot of activity concerning more generally the convex hull of columns of TU-matrices which are studied under the (ambiguous) name of {\em root polytopes} or, in the centrally-symmetric case, {\em generalized symmetric edge polytopes} \cite{RegMat1, RegMat2, RegMat3, RegMat4}. As \Cref{translation} shows, any unimodular polytope is isomorphic to such a root polytope (but not vice versa). 
\end{rmk}

Let us explain the direct connection between unimodular polytopes and polytopal TU-matrices in more detail. We abbreviate $\conv(M)$ for the convex hull of the columns of a matrix $M$.

As it is natural in this context, let us recall that an $m\times n$-matrix of rank $m$ is called {\em unimodular} if every minor of order $m$ is in $\{-1,0,1\}$ (i.e., the columns form a unimodular point configuration). In particular, any TU $m \times n$-matrix of rank $m$ is unimodular.

\begin{proposition}
Let $M$ be a polytopal unimodular $(d+1) \times n$-matrix. Then $\conv(M) \subset \R^{d+1}$ is a unimodular polytope of dimension $d$ with $n$ vertices. 

On the other hand, every unimodular polytope of dimension $d$ is isomorphic to $\conv((I_{d+1} | B))$ where $B$ and $(I_{d+1}|B)$ are polytopal TU-matrices with column sums $1$. 
\label{translation}
\end{proposition}

\begin{proof} When $M$ is a polytopal unimodular matrix, any $d+1$ linearly independent set of columns of $M$ is a lattice basis of $\Z^{d+1}$, therefore, the vertex set of a $d$-dimensional unimodular simplex. Hence, $\conv(M)$ is a unimodular polytope of dimension $d$.

On the other hand, let $P$ be a unimodular polytope of dimension $d$. After a suitable isomorphism, we may assume that $P \subset \R^d \times \{1\}$ and that the first $d+1$ vertices of $P$ form a lattice basis of $\Z^{d+1}$. Let $M$ be the unimodular matrix with the vertices of $P$ as columns. Multiplying from the left with the inverse $R^{-1}$ of the matrix $R$ of the first $d+1$ columns yields a unimodular matrix $M' =(I_{d+1}|B)$ with $P \cong \conv(M')$. As is well-known, $B$ and $M'$ are TU-matrices (note that the $(d+1)$-minors of $M'$ with $k$ columns in $I_{d+1}$ equal $\pm$ the $(d+1-k)$-minors of $B$). 
As $(0 \cdots 0 \, 1)$ evaluates to $1$ on every column of $M$, $(1 \cdots 1\, 1) = (0 \cdots 0 \,1) R$ evaluates to $1$ on every column of $M'$.
\end{proof}

\begin{rmk}
 Let us note the following observation from the proof of \Cref{translation}: every unimodular polytope can be represented by a unimodular matrix (not necessarily a TU-matrix) with a row vector with all ones.
\end{rmk}

\subsection{The maximal number of lattice points/vertices of unimodular polytopes}

In the sense of the hierarchy of covering properties of lattice polytopes given in \cite[Section~1.2.5]{Christian} and \cite[Proposition~9.3.20]{Triangulationsbook}, unimodular polytopes are at the very top and thus among the most interesting natural classes of lattice polytopes. Hence, they form a perfect testing ground for the main conjectures in lattice polytope theory such as the unimodality questions in Ehrhart theory \cite{braun}. For instance, while it was just recently shown that Ehrhart $h^*$-polynomials of unimodular polytopes do not have to be log-concave \cite{ML}, the major question whether they can be non-unimodal is still wide open. In this paper, we illustrate how the existing strong results on TU-matrices can be helpful in improving our understanding of unimodular polytopes by giving a sharp upper bound on their number of vertices.

Note that the only lattice points in a unimodular polytope are its vertices. In fact, unimodular polytopes are isomorphic to {\em 0/1-polytopes}, i.e., they are isomorphic to lattice subpolytopes of $[0,1]^{\dim(P)}$. 
0/1-polytopes are classified up to dimension $5$ \cite{Aichholzer} and available in the database PolyDB \cite{PolyDB}. Using Polymake \cite{Polymake} and OSCAR \cite{OSCAR,OSCAR-book} we get the following complete classification of unimodular polytopes up to dimension $5$.

\begin{proposition}
    For each dimension $d$ there are the following number of isomorphism classes of unimodular polytopes of dimension $d$:
    
    \begin{center}
    \begin{tabular}{c||c|c|c|c|c}
        dimension $d$&1&2&3&4&5\\\hline
        \text{\# unimodular polytopes}&1&2&4&13&38
    \end{tabular}
    \end{center}
    \label{database}
\end{proposition}

From Example~\ref{ex:ex} one might guess that the maximal number of vertices is always attained by the product of two unimodular simplices, however, this is wrong in dimension $d=4$.

\begin{ex}
Note that the number of vertices of $\Delta_2 \times \Delta_2$ is $9$. However, the convex hull of the columns of the following $4\times 10$-matrix (note that the matrix is not TU) is a $4$-dimensional unimodular 0/1-polytope with $10$ vertices:\vspace{1ex}

\begin{center}
$\begin{pmatrix}
0 & 0 & 0 & 0 & 1 & 1 & 1 & 1 & 1 & 1 \\\\
0 & 1 & 1 & 1 & 0 & 0 & 0 & 1 & 1 & 1 \\\\
1 & 0 & 1 & 1 & 0 & 1 & 1 & 0 & 0 & 1 \\\\
1 & 1 & 0 & 1 & 1 & 0 & 1 & 0 & 1 & 0
\end{pmatrix}$
\end{center}

\label{ex:4}
\vspace{1ex}

Unimodularity can be checked by considering all subsets of the columns of size $5$, subtracting one of the columns from the others, and verifying that the determinant of these four non-zero vectors equals $\pm1$. 
\end{ex}

As an application of \Cref{theorem:tu} (with $m=d+1$), we get a sharp bound on the number of vertices of unimodular polytopes that shows that this exceptional behavior only happens in dimension $4$.

\begin{corollary}
    The number of vertices of a $d$-dimensional unimodular polytope $P$ is bounded from above by $10$ if $d=4$ and by $\lfloor \frac{(d+2)^2}{4} \rfloor$ otherwise. The bound is attained by \Cref{ex:4} if $d=4$, by $\Delta_{\frac{d}{2}} \times \Delta_{\frac{d}{2}}$ if $d\not=4$ even, and by $\Delta_{\frac{d+1}{2}} \times \Delta_{\frac{d-1}{2}}$ if $d$ is odd.
    \label{main}
\end{corollary}

It appears that this result may be the first instance that Seymour's decomposition theorem is successfully used in the setting of lattice polytopes. It is reasonable to believe that more results on unimodular polytopes may be obtained in this way\footnote{In fact, already shortly after this paper appeared as a preprint, Seymour's decomposition theorem has also been successfully applied for subclasses of smooth Fano polytopes \cite{Binnan}.}.

\medskip

\begin{rmk}
    It seems a very challenging open question to find similar sharp results on the number of lattice points respectively vertices of {\em $\Delta$-modular polytopes}: lattice polytopes $P$ where $\Delta$ is the maximal normalized volume of any full-dimensional lattice simplex contained in $P$. Finding sharp bounds for the number of lattice points is a hard problem even in the case for lattice simplices. This question is motivated by recent papers \cite{glanzer, boundedsubdeterminants, Oxley, Lee, averkov2024maximal, Paat, kriepke2025size} on the size of integer programs with bounded subdeterminants.
    \label{bounded}
\end{rmk}

\section{The proof of the column number bound for polytopal TU-matrices}

\label{sec:proof}

The remainder of this paper contains the proof of \Cref{theorem:tu} using Seymour's decomposition theorem. First, we will look at network matrices and their transposes, before dealing with the inductive step using so-called $k$-sums and $\Delta$-sums. As it turns out, instead of focusing on column sums equal to $1$ it is often beneficial to relax this to the more general condition that the column sums are just positive and odd. For the induction to work, it will also be necessary to consider more general linear functionals than just summing the column entries.

Throughout the paper we use $\pm$ to denote $+$ or $-$.

\subsection{Network matrices}
\label{netw}

 We say $M$ is an {\em $A_0 \times A$-network matrix} \cite{Schrbook} (or {\em tree matrix} \cite{Seymour-survey}) if it is associated to the set of arcs $A_0$ of a directed tree $T$ with vertex set $V$ and to the set of arcs $A$ of a digraph $D$ on $V$ in the following way. Let $e$ be an arc of $D$ with vertices $s \rightarrow t$, and let $H$ be the unique path in the underlying undirected graph of $T$ from $s$ to $t$. Then the column of $M$ corresponding to $e \in A$ contains at the row position corresponding to an arc $e'$ of $T$ the entry $1$ if $e'$ is passed with the same orientation in $H$, $-1$ if it is passed with the reverse orientation in $H$, and $0$ if this edge is not used in $H$.

\begin{proposition}
    Let $M$ be an $A_0 \times A$-network matrix such that all columns are pairwise different and have positive odd column sums. Then $|A|$ is bounded from above by $\frac{(|A_0|+1)^2}{4}$. Equality implies that $|A_0|$ is odd and the undirected graph underlying the digraph $D$ is a complete bipartite graph with two parts of the same size $\frac{|A_0|+1}{2}$.
    \label{prop:network}
\end{proposition}

\begin{proof}
We use notation and results from Schrijver's book \cite{Schrbook} (see the proof of equation (34) on page 277). Let $N$ be the $V \times A$-incidence matrix of $D$ and $L$ be the $V \times A_0$-incidence matrix of $T$. Moreover, let $\tilde{N}$ and $\tilde{L}$ be given by deleting the last row in $N$, respectively $L$. Note that $L M = N$, $\tilde{L} M = \tilde{N}$, and $\tilde{L}$ is invertible. Let us define $\1 := (1, \ldots, 1) \in \Z^{A_0}$ and let $z=(z_1, \ldots, z_A) \in \Z^A$ be such that $\1^\top M = z^\top$. Thus, $\1^\top\tilde{L}^{-1} \tilde{N} = z^\top$. Define $u^\top \in \Z^{V}$ as $\1^\top \tilde{L}^{-1}$ extended by $0$ in the last coordinate. Then we see $u^\top N = z^\top$.

Let $G$ be the undirected graph underlying $D$. We will show that $G$ is bipartite by observing that all cycles must be even. For this, let $v_1, \ldots, v_r$ be vertices in $V$ forming a cycle in $G$ and assume that $r$ is odd. We set $v_{r+1} := v_1$. Denoting the successive corresponding arcs in $A$ by $a_1, \ldots, a_r$, we get from $u^\top N = z^\top$ that $u_{v_{i+1}} = u_{v_i} \pm z_{a_i}$ for $i=1, \ldots, r$. Hence, $u_{v_{r+1}} = u_{v_1} \pm z_{a_1} \pm \cdots \pm z_{a_r}$, so $\pm z_{a_1} \pm \cdots \pm z_{a_r} = 0$. As by our assumption each entry of $z$ is odd, and $r$ is odd, we derive at a contradiction.

It is well-known that a bipartite graph $G$ with $|V|$ vertices can have at most $\frac{|V|^2}{4}$ edges, with equality if and only if $|V|$ is even and $G$ is a complete bipartite graph with two parts of size $\frac{|V|}{2}$. Now, use $|V|=|A_0|+1$.
\end{proof}

\subsection{Transposes of network matrices}

In this quite different setting there is even a {\em linear} bound on the number of rows of a network matrix in terms of its columns.

\begin{theorem}
    Let $M$ be an $A_0 \times A$-network matrix. Then the set of pairwise different rows of $M$ with positive row sums is bounded from above by $3|A|-3$. Moreover, if $|A| \ge 3$, then the set of pairwise different rows of $M$ with positive odd row sums is bounded from above by $3|A|-5$. 
    \label{theorem:transpose}
\end{theorem}

Note that for $|A| \ge 7$ we have $3|A|-5 \le \frac{(|A|+1)^2}{4}$.

We remark that the bound $3|A|-3$ is already known \cite[Theorem~3.1]{bonin}, where an elegant proof is given in the setting of cographical geometries. Here, we give a different proof that allows for the desired improvement of the bound in the case of odd row sums. For this, let us recall that $[m] := \{1, \ldots, m\}$. 

We fix an undirected tree $T$ that contains $m$ undirected paths $P_1, \ldots, P_m$. For each edge $e$ of $T$ we define 
\[I_e := \{j \in [m] \;:\; e \in E(P_j)\}.\]
If $I_e \not=\emptyset$, i.e., $e$ is contained in at least one of the paths, then we call $I_e$ a {\em pattern} of $P_1, \ldots, P_m$. One should think of $I_e$ as a vector in $\{0,1\}^m$ with $1$ at the $j$th position if $e$ is in $P_j$. As an example, for $m=2$ with two paths $P_1, P_2$ we have at most three distinct patterns, namely, $\{1\},\{2\},\{1,2\}$. Note that different edges may have the same pattern. The reader is invited to draw a picture of three paths in a tree and to compute their set of distinct patterns before proceeding.

\begin{proposition}
In this situation, for $m \ge 2$ the number of distinct patterns is at most $3m-3$. Moreover, for $m \ge 3$ the number of distinct patterns of odd size is at most $3m-5$.
\label{prop:graph}
\end{proposition}

\begin{proof}

Let $1 \le i \le m-1$. We consider the first $i$ paths $P_1, \ldots, P_i$. We will denote the subgraph $P_1 \cup \cdots \cup P_i$ of $T$ by $H$. Let $H$ consist of the connected components $B_1, \ldots, B_{r_i}$.

Now, we consider the situation with one more path $P_{i+1}$. Clearly, if an edge of $P_{i+1}$ is not contained in $H$, then we get a `new' pattern $\{i+1\}$. From now on, we are interested in edges of $H$. We denote the `new' pattern of an edge $e$ of $H$ by $\hat{I}_e \subseteq \{1, \ldots, i+1\}$. If $e$ is contained in $P_{i+1}$, then $\hat{I}_e = I_e \cup \{i+1\}$; otherwise, $\hat{I}_e = I_e$. 

If $P_{i+1}$ intersects (even if just in a vertex) one of the connected components, say $B_1$, we call this a {\em merge}. In this situation, we will show that it can happen at most twice that for an 'old' pattern $I_e$ both $I_e$ {\em and} $I_e \cup \{i+1\}$ are `new' patterns. So, let us assume that there are two edges $e$ and $e'$ of $B_1$ with $I_e = I_{e'}$ such that $\hat{I}_{e} = I_{e} \cup \{i+1\}$ (so $e \in P_{i+1}$) and $\hat{I}_{e'} = I_{e'}=I_e$ (so $e' \not\in P_{i+1}$). Let us remark that as $B_1$ is a tree, all edges of $B_1$ with the same pattern form a subpath of $B_1$. Hence, the path containing $e$ and $e'$ contains between them one of the two edges of $B_1$ where $P_{i+1}$ enters or leaves the tree $B_1$. We deduce that there are at most two `old' patterns with this behavior, as desired. We conclude that in such a merge the number of distinct patterns that are different from $\{i+1\}$ increases at most by two. Overall, the number of distinct patterns after a merge increases at most by three. 

As after the very first merge -- where just two paths intersect -- we have at most $3$ distinct patterns, we see that if $\ell$ denotes the number of all occurring merges for $P_1, \ldots, P_m$, the number of distinct patterns for $P_1, \ldots, P_m$ is bounded from above by $3 \ell$.

It thus remains to show $\ell \le m-1$. For this, we define the {\em merge set} and initialize it as $\{\{1\}, \{2\}, \ldots, \{m\}\}$. Whenever in our above procedure a merge happens we also unite the elements of the merge set containing the indices of the corresponding paths involved. It is illustrative to consider this by an example. For instance, for $m=6$ and $i=4$, let us assume that $H$ consists of the two connected components $B_1 := P_1 \cup P_2$ and $B_2 := P_3 \cup P_4$ with merge set $\{\{1,2\},\{3,4\},\{5\},\{6\}\}$. Next, we assume that $P_5$ intersects with $B_1$ and $B_2$. In this case, the first merge of $P_5$ with $B_1$ would lead to the merge set $\{\{1,2,5\},\{3,4\},\{6\}\}$, and the next merge of $P_5$ with $B_2$ changes the merge set to $\{\{1,2,3,4,5\},\{6\}\}$. It remains to notice that in each merge the size of the merge set decreases by one, so there can be at most $m-1$ merges, thus, $\ell\le m-1$. 

Let us consider the additional statement. Let $m \ge 3$, and let us assume that we have $3m-3$ distinct patterns for $P_1, \ldots, P_m$, so $\ell = m-1$ merges occurred. Let us consider the last merge that occurred (because of $\ell \ge 2$, this is not the first merge). Here, in above notation there must necessarily exist two edges $e$ such that $I_e \cup \{m\}$ and $I_e$ are `new' patterns. Hence, there must exist at least two distinct patterns of even size, so at most $3m-5$ distinct patterns of odd size. Analogously, if we have $3m-4$ distinct patterns, also at least $m-1$ merges must have occurred. Again considering the last merge we see that there has to be at least one edge $e$ with patterns $I_e \cup \{m\}$ and $I_e$, one of them necessarily of even size. Hence, there are at most $3m-5$ distinct patterns of odd size.
\end{proof}

\begin{proof}[Proof of \Cref{theorem:transpose}]
Let $m := |A|$. We consider the underlying undirected graph of $T$ and the $m$ paths $P_1, \ldots, P_m$ on $T$ associated to the $m$ edges of $A$. We define $\bar{M}$ as the $|E(T)| \times m$-matrix with entries in $\{0,1\}$ given by reducing any entry of $M$ modulo $2$. Note that an edge $e$ of $T$ lies in $E(P_j)$ for $j \in [m]$ if and only if the corresponding entry in $\bar{M}$ is equal to $1$. Hence, $I_e$ equals the `support set' of the corresponding row in $\bar{M}$. Let us also observe that, as each row vector of $M$ has only entries in $\{-1,0,+1\}$, having odd row sum in $M$ is equivalent to odd row sum in $\bar{M}$. 

In order to apply \Cref{prop:graph} it suffices to show that distinct rows in $M$ with positive row sum stay distinct in $\bar{M}$. Assume that there are two arcs $e$ and $e'$ in $A_0$ whose corresponding rows in $M$ have positive rows sum that map to the same row in $\bar{M}$, i.e., $I_e = I_{e'}$. If there would be two paths $P_i$ and $P_j$ in $T$ (for $i,j \in I_e=I_{e'},\, i\not=j$) such that $M_{e,i}=M_{e,j}$ but $M_{e',i} \not=M_{e',j}$, then this would lead to a contradiction (as the corresponding $2 \times 2$-submatrix of $M$ has non-zero determinant and only entries in $\{-1,1\}$, its determinant would be equal to $\pm 2$). Hence, the $e$-row of $M$ equals $\pm$ the $e'$-row of $M$. As both rows have positive row sum, we see that they must be equal.
\end{proof}

\subsection{On $k$-sums and $\Delta$-sums of TU-matrices}
\label{k-sums}

Let us recall some operations on TU-matrices, where we follow the notation and definitions given in Seymour's survey article \cite{Seymour-survey}. By convention, vectors are considered as column vectors.

\medskip

In this subsection, we assume that $A$ is an $m_1 \times n_1$-matrix and $B$ is an $m_2 \times n_2$-matrix. 

\begin{definition}
    We say 
    \[\begin{pmatrix}A & 0\\ 0 & B\end{pmatrix}\]
    is the {\em $1$-sum} of TU-matrices $A$ and $B$.
    \label{1-sum}
\end{definition}

\begin{definition}
    We say 
    \[\begin{pmatrix}A & u v^\top\\ 0 & B\end{pmatrix}\]
    is the {\em $2$-sum} of TU-matrices $(A | u)$ and $\begin{psmallmatrix}v^\top\\B\end{psmallmatrix}$.
    \label{2-sum}
\end{definition}

Note that adding a standard basis vector to a TU-matrix as a new column is a $2$-sum construction (e.g., choose $A=1=u$ in order to add the first standard basis vector).

As in \cite{Seymour-survey}, we define the {\em size} of an $m\times n$-matrix as $m+n$. The assumptions on the sizes in the two following definitions correspond precisely to the condition that the size of each `factor' TU-matrix in a $k$-sum must be strictly smaller than the size of the whole TU-matrix.

\begin{definition}
    If the sizes of $A$ and $B$ are each at least 4, we say 
    \[\begin{pmatrix}A & C\\ 0 & B\end{pmatrix}\]
    is a {\em $3$-sum} of TU-matrices $(A | u_1 | u_2 | u_3)$ and $\begin{psmallmatrix}v^\top_1\\v^\top_2\\v^\top_3\\B\end{psmallmatrix}$, where 
    $u_1 + u_2 + u_3 = 0$ and $v_1 + v_2 + v_3 = 0$, and every column of $C$ is one of $\pm u_1, \pm u_2, \pm u_3, 0$ and 
    every row of $C$ is one of $\pm v^\top_1, \pm v^\top_2, \pm v^\top_3, 0$.\label{3-sum}
\end{definition}

We call the following second type of a $3$-sum, also described in \cite{Seymour-survey}, $\Delta$-sum. This notion was suggested by an anonymous referee as it is related to the Delta-Wye transformation.

\begin{definition}
    If the sizes of $A$ and $B$ are each at least 4, we say 
    \[\begin{pmatrix}A & u' v^\top\\ v' u^\top & B\end{pmatrix}\]
    is a {\em $\Delta$-sum} of TU-matrices $\begin{pmatrix}A & u' & u'\\u^\top & 0 & x\end{pmatrix}$ and $\begin{pmatrix}v^\top & 0 & x\\B & v' & v'\end{pmatrix}$ with $x = \pm 1$.\label{4-sum}
\end{definition}

Let us introduce some more terminology. 

\begin{definition}
Let $M$ be an $m\times n$-matrix TU-matrix. We say $M$ is {\em $w$-valued} for $w \in \Z^n$ if there exists $f \in (\Z^{m})^* \setminus \{0\}$ such that $f$ evaluates on the $j$th column of $M$ to $w_j$ for any $j\in[n]$. 
\label{definition:odd}
\end{definition}

In order for the induction step in the proof of \Cref{theorem:tu} to work, one needs to modify the functional evaluating to $w$ accordingly.
   
\begin{lemma}
Let $M$ be a TU $m \times n$-matrix that is $w$-valued with functional $f$ and has distinct columns such that $M$ is a $k$-sum (for some $k \in [3]$) or a $\Delta$-sum in the notation of Definitions~\ref{1-sum}-\ref{4-sum}. Let $w=(w_1,w_2)$ for $w_1 \in \Z^{n_1}$ and $w_2 \in \Z^{n_2}$.
\begin{enumerate}
    \item If $M$ is a $2$-sum, then $\begin{psmallmatrix}v^\top\\B\end{psmallmatrix}$ has distinct columns, 
    and there is a functional $f' \in (\Z^{m_2+1})^*$ such that $\begin{psmallmatrix}v^\top\\B\end{psmallmatrix}$ is $w_2$-valued with $f'(e_j) =f(e_{m_1+j-1})$ for $j=2, \ldots, m_2+1$.
    \item If $M$ is a $3$-sum and $\pm v^\top_1$ and $\pm v^\top_2$ appear as rows of $C$, then $\begin{psmallmatrix}v^\top_1\\v^\top_2\\B\end{psmallmatrix}$ has distinct columns, and there is a functional $f' \in (\Z^{m_2+2})^*$ such that $\begin{psmallmatrix}v^\top_1\\v^\top_2\\B\end{psmallmatrix}$ is $w_2$-valued with $f'(e_j)=f(e_{m_1+j-2})$ for $j=3, \ldots, m_2+2$.
    \item If $M$ is a $\Delta$-sum, then $\begin{psmallmatrix}A\\u^\top\end{psmallmatrix}$ and $\begin{psmallmatrix}v^\top\\B\end{psmallmatrix}$ have each distinct columns, and 
    \begin{itemize}
        \item there is a functional $f' \in (\Z^{m_1+1})^*$ such that $\begin{psmallmatrix}A\\u^\top\end{psmallmatrix}$ is $w_1$-valued with $f'(e_j)=f(e_j)$ for $j=1, \ldots, m_1$, and
        \item there is a functional $f'' \in (\Z^{m_2+1})^*$ such that $\begin{psmallmatrix}v^\top\\B\end{psmallmatrix}$ is $w_2$-valued with $f''(e_j)=f(e_{m_1+j-1})$ for $j=2, \ldots, m_2+1$.
    \end{itemize}
    \end{enumerate}
\label{lem:odd}
\end{lemma}

\begin{proof}
Let $M$ be a $2$-sum. Assume that two columns $j_1, j_2$ of $\begin{psmallmatrix}v^\top\\B\end{psmallmatrix}$ would be equal. Then also the $(n_1+j_1)$-th and the $(n_1+j_2)$-th column of $M$ would be equal to each other, a contradiction. Let $f=(f_1,f_2)$ with $f_1 \in (\Z^{m_1})^*$ and $f_2 \in (\Z^{m_2})^*$. We observe that for $j \in [n_2]$, the $f$-value of the $(n_1+j)$-th column of $M$ equals the sum of $f_1(u) \cdot v_j$ and the $f_2$-value of the $j$-th column of $B$. For $f'=(f_1(u),f_2) \in (\Z^{1+m_2})^*$ the statements follow. 

Let $M$ be a $3$-sum. Assume that two columns $j_1$ and $j_2$ of $\begin{psmallmatrix}v^\top_1\\v^\top_2\\B\end{psmallmatrix}$ would be equal. Then $v_1 + v_2 + v_3 = 0$ implies that the $(n_1+j_1)$-th and the $(n_1+j_2)$-th column of $M$ would be equal to each other, a contradiction. Here, we use the coordinate-wise notation $f=(f_1, \ldots, f_m)$. Let us define $f'$ as follows. For $\ell \in [3]$, we define $I_\ell \subseteq [m_1]$ as the set of indices of rows of $C$ that are equal to $v^\top_\ell$ and $I'_\ell \subseteq [m_1]$ as the set of indices of rows of $C$ that are equal to $-v^\top_\ell$. Now, for $\ell \in [2]$ we define $f'_\ell := \sum_{t \in I_\ell} f_t - \sum_{t \in I'_\ell} f_t - \sum_{t \in I_3} f_t + \sum_{t \in I'_3} f_t$. For $i \in \{3, \ldots, m_2+2\}$ we define $f'_i := f_{m_1+i-2}$. It is direct to check that this choice of $f'$ satisfies the desired conditions.

Let $M$ be a $\Delta$-sum. Let $f=(f_1,f_2)$ with $f_1 \in (\Z^{m_1})^*$ and $f_2 \in (\Z^{m_2})^*$. For $j \in [n_1]$ the $f$-value of the $j$-th column sum of $M$ equals the sum of the $f_1$-value of the $j$-th column sum of $A$ and $f_2(v') \cdot u_j$. For $(f_1,f_2(v')) \in (\Z^{m_1+1})^*$ we get that $\begin{psmallmatrix}A\\u^\top\end{psmallmatrix}$ is $w_1$-valued. Assume that two columns $j_1$ and $j_2$ of $\begin{psmallmatrix}A\\u^\top\end{psmallmatrix}$ would be equal to each other. Then also the $j_1$-th and the $j_2$-th column of $M$ would be equal to each other, a contradiction. Similarly, one proves the corresponding statement for $\begin{psmallmatrix}v^\top\\B\end{psmallmatrix}$.
\end{proof}

\subsection{Technical inequalities regarding the bound in \Cref{theorem:tu}}\

We abbreviate $g(x) := \frac{(x+1)^2}{4}$. We define $h \,:\, \Z_{\ge 1} \to \Z_{\ge 1}$ as $h(x) := \lfloor g(x) \rfloor$ if $x \not= 5$, and $h(5):= 10$. Note that $h(5)=g(5)+1$, $h(x)=g(x)$ if $x\not=5$ is odd and $h(x)=g(x)-\frac{1}{4}$ if $x$ is even. This implies for any $x \in \Z_{\ge 1}$ the estimates
\begin{equation}
    h(x) \le g(x+1) \text{ and } g(x) - \frac{1}{4} \le h(x).
\label{neue-eq}
\end{equation}

We will need the following auxiliary result. 

\begin{lemma}\
\begin{enumerate}
    \item Let $x,y \in \Z_{\ge 1}$. Then $h(x) + h(y) \le h(x+y)$.
    \item Let $x \in \Z_{\ge 2}$, $y \in \Z_{\ge 1}$. Then $h(x) + h(y+1) \le h(x+y)$.
    \item Let $x \in \Z_{\ge 3}$, $y \in \Z_{\ge 1}$. Then $h(x) + h(y+2) \le h(x+y)$ 
    except if $(x,y)$ is in 
    \[\{(3,1),(3,3),(5,1)\}.\]
    \item Let $x,y \in \Z_{\ge 2}$. Then $h(x+1) + h(y+1) \le h(x+y)$ 
    except if $(x,y)$ is in 
    \[\{(2,2),(2,4),(4,2)\}.\]
    \end{enumerate}
    \label{extralemma}
\end{lemma}

All inequalities in the following proof can be easily checked using any computer algebra system e.g. \cite{wolfram}.

\begin{proof}
    (1) By \eqref{neue-eq} the statement follows if $(g(x)+1)+(g(y)+1) \le g(x+y)-\frac{1}{4}$, which is equivalent to $(x+y+1)^2-(x+1)^2-(y+1)^2 \ge 9$, which holds except if $y < \frac{5}{x}$. Hence, by symmetry we are left with checking $h(x) + h(y) \le h(x+y)$ for $x,y \le 4$.
   
   (2) By \eqref{neue-eq} the statement follows if $(g(x)+1)+(g(y+1)+1) \le g(x+y)-\frac{1}{4}$, which is equivalent to $(x+y+1)^2-(x+1)^2-(y+2)^2 \ge 9$, which, as $x \ge 2$, holds except if $y < \frac{13}{2(x-1)}$. Hence, we are left with checking $h(x) + h(y+1) \le h(x+y)$ for $2 \le x \le 7$ and $1 \le y \le 6$.

   (3) By \eqref{neue-eq} the statement follows if $(g(x)+1)+(g(y+2)+1) \le g(x+y)-\frac{1}{4}$, which is equivalent to $(x+y+1)^2-(x+1)^2-(y+3)^2 \ge 9$, which, as $x \ge 3$, holds except if $y < \frac{9}{x-2}$. Hence, we are left with checking $h(x) + h(y+2) \le h(x+y)$ for $3 \le x \le 10$ and $1 \le y \le 8$.
    
   (4) By \eqref{neue-eq} the statement follows if $(g(x+1)+1)+(g(y+1)+1) \le g(x+y)-\frac{1}{4}$, which is equivalent to $(x+y+1)^2-(x+2)^2-(y+2)^2 \ge 9$, which, as $x,y \ge 2$, holds except if $y < \frac{x+8}{x-1}$, which is at most $10$. Hence, by symmetry we are left with checking $h(x+1) + h(y+1) \le h(x+y)$ for $2 \le x,y \le 9$.
\end{proof}

\subsection{Proof of \Cref{theorem:tu}}

Let us note that $M$ is polytopal if and only if $M$ is $\1$-valued. We say $M$ is {\em prepared} if $M$ is polytopal and all columns are pairwise distinct. 
    
Let $M$ be a prepared TU-matrix with $m$ rows. We will prove that the number of columns of $M$ cannot exceeds $h(m)$ by induction on the size. 

From the TU-property it follows that one can multiply $M$ from the left by a unimodular transformation in $\text{GL}(\Z,m)$ and permute the columns in a way such that the new matrix is of the form 
\[\begin{pmatrix}I_r & M'\\0 & 0\end{pmatrix}\]
where $(I_r | M')$ is also a prepared TU-matrix.  
By the induction hypothesis, we may assume that $M = (I_m | M')$, where 
also $M'$ is prepared (hence, by our assumption $M'$ does not contain any standard basis vector as a column). In particular, $M$ has full rank $m$, and $M$ and $M'$ have constant column sum $1$. 

By the following remark we can assume that $m \ge 7$.

\begin{rmk}
By \Cref{translation} any prepared TU-matrix with $m$ rows of rank $m$ defines an $(m-1)$-dimensional unimodular polytope. Hence, for $m \le 6$ we can use \Cref{database} to determine the maximal number of columns and verify our statement. 
\label{fewrows}
\end{rmk}

 We denote the number of columns of $M'$ by $n$, so $M'$ is an $m\times n$-matrix, and $M$ is an $m \times (m+n)$-matrix and has size $2m+n$. We have to show $n \le h(m)-m$. 

\medskip
Let us give one more definition \cite{Seymour-survey}. We call any $5\times5$-matrix {\em sporadic} if it can be obtained from one of the following two matrices by permuting rows and columns and scaling some rows and columns by $-1$:

\[\begin{pmatrix}
    -1&1&0&0&1\\1&-1&1&0&0\\0&1&-1&1&0\\0&0&1&-1&1\\1&0&0&1&-1
\end{pmatrix},\quad
\begin{pmatrix}
    1&1&0&0&1\\0&1&1&0&1\\0&0&1&1&1\\1&0&0&1&1\\1&1&1&1&1
\end{pmatrix}
\]

Now, we make use of the following fundamental result (this version for TU-matrices can be found in \cite{Seymour-survey}, see \cite{Seymour-original} for the original version for regular matroids).

\begin{theorem}[Seymour's decomposition theorem]
Up to permutation of rows and columns, each TU-matrix is a network matrix, the transpose of a network matrix, a sporadic matrix, or a $k$-sum (for $k \in [3]$) or $\Delta$-sum of TU-matrices.
\end{theorem}

If $M'$ is a network matrix it is easy to see that also $M$ is a network matrix. Let us show the corresponding statement for transposes of network matrices. Let $M'$ be the transpose of a network matrix, for which we will use the notation $A',A'_0,T'$ as defined in Subsection~\ref{netw}. For each arc $e' \in A'$ with vertices $s_{e'} \rightarrow t_{e'}$, we define a new vertex $v_{e'}$, a new arc $a_{e'}$ with vertices $v_{e'} \to s_{e'}$, and a new arc $e$ with vertices $v_{e'} \to t_{e'}$. Then $A_0 := \{a_{e'} \,:\, e' \in A'\} \cup A'_0$ is the set of arcs of a directed tree $T$ and $A := \{e \,:\, e' \in A'\}$ the set of arcs of a digraph $D$. Then $M = (I_{|A'|} | M')$ is the transpose of the associated network matrix. Hence, whenever $M'$ is a network matrix, respectively, the transpose of a network matrix, we get our desired bound by applying \Cref{prop:network} to $M$, respectively \Cref{theorem:transpose} to $M^\top$. If $M'$ is a sporadic matrix, then $n = 5 = 10-5 = h(m)-m$.

Now, Seymour's decomposition theorem applied to $M'$ implies that after a permutation of rows and columns we may assume that $M'$ is a $k$-sum for $k\in [3]$ or $\Delta$-sum, where the $m_1 \times n_1$-matrix $A$ and $m_2 \times n_2$-matrix $B$ are given as in \Cref{k-sums}. Note that $m_1+m_2 = m$ and $n_1+n_2 = n$. We will consider these four cases separately.

\medskip

Let $M'$ be a $1$-sum. By the induction hypothesis $\left(e_1 \cdots e_{m_1} | A\right)$ and $\left(e_{m_1+1} \cdots e_{m} | B\right)$ are prepared TU-matrices (as they are submatrices of smaller sizes) such that $m_1+n_1 \le h(m_1)$ and $m_2+n_2 \le h(m_2)$. Hence, $n =n_1+n_2 \le h(m_1)-m_1+h(m_2)-m_2 \le h(m_1+m_2) - (m_1+m_2) = h(m_1+m_2) - m$ by \Cref{extralemma}(1).

\medskip

Let $M'$ be a $2$-sum. As $M'$ is prepared and contains no standard basis vectors as column vectors, we have $m_1 \ge 2$. Note that 
even $m_1 \ge 3$ holds as no vector in $\{-1,0,1\}^2$ has odd column sum except for $\pm$ a standard basis vector. 
Again, $(e_1 \cdots e_{m_1} | A)$ is a prepared TU-matrix, hence, by the induction hypothesis, $m_1+n_1\le h(m_1)$. By \Cref{lem:odd}(1) applied to $M'$ (recall that $w=\1$), 
$$\begin{pmatrix}0 & v^\top \\ I_{m_2} & B\end{pmatrix}$$ is a prepared TU $(m_2+1) \times (m_2+n_2)$-matrix. Here, note that $\begin{psmallmatrix} v^\top\\B\end{psmallmatrix}$ cannot contain a standard basis vector $e_j$ with $j \in \{2, \ldots, m_2+1\}$ as a column as otherwise by the definition of a $2$-sum the matrix $M'$ would contain $e_{m_1+j-1}$ as a column. The induction hypothesis implies $m_2+n_2 \le h(m_2+1)$. Therefore, $n =n_1+n_2 \le h(m_1)-m_1 + h(m_2+1)-m_2 \le h(m_1+m_2)-(m_1+m_2)$ by \Cref{extralemma}(2).

\medskip

Let $M'$ be a $3$-sum. As for the case of a $2$-sum, we have $m_1 \ge 3$ and $m_1+n_1\le h(m_1)$. As we otherwise would be in the case of a $1$-sum, $C$ has a non-zero row $w^\top$. If all non-zero rows of $C$ would be of the form $\pm w^\top$, this would be the case of a $2$-sum. Hence, we may assume that there are at least two non-zero and non-centrally-symmetric rows of $C$, say, of the form $\pm v^\top_1$ and $\pm v^\top_2$. By \Cref{lem:odd}(2), 
$$\begin{pmatrix}0 & v^\top_1 \\ 0 & v^\top_2\\ I_{m_2} & B\end{pmatrix}$$ is a prepared $(m_2+2)\times (m_2+n_2)$ TU-matrix. Again, note that $$\begin{pmatrix} v^\top_1\\v^\top_2\\B\end{pmatrix}$$ cannot contain a standard basis vector $e_j$ with $j \in \{3, \ldots, m_2+2\}$ as a column as otherwise by the definition of a $3$-sum the matrix $M'$ would contain $e_{m_1+j-2}$ as a column. Therefore, the induction hypothesis implies $m_2+n_2 \le h(m_2+2)$. Hence, $n =n_1+n_2 \le h(m_1)-m_1 + h(m_2+2)-m_2 \le h(m_1+m_2)-(m_1+m_2)$ by \Cref{extralemma}(3); note that all three exceptional cases cannot occur as they satisfy $m = m_1+m_2 \le 6$.

\medskip

Let $M'$ be a $\Delta$-sum. In the same way as above, 
$$\begin{pmatrix}I_{m_1} & A \\ 0 & u^\top\end{pmatrix} \text{ and }
\begin{pmatrix}0 & v^\top\\I_{m_2} & B\end{pmatrix}$$ are prepared TU-matrices by \Cref{lem:odd}(3). Hence, the induction hypothesis implies $n_1 \le h(m_1+1)-m_1$ and $n_2 \le h(m_2+1)-m_2$. Therefore, \Cref{extralemma}(4) implies $n =n_1+n_2 \le h(m_1+1)-m_1 + h(m_2+1)-m_2 \le h(m_1+m_2)-(m_1+m_2)$ if $m_1\ge 2$ and $m_2 \ge 2$. 

Let us assume that $m_1 = 1$. If there would be an $i \in \{1, \ldots, n_1\}$ such that 
the $i$th coordinate of $u^\top$ would be $0$, then all coordinates except possibly the first one of the $i$th-column vector of $M'$ would be $0$, a contradiction to $M'$ having column sums $1$ and not containing any standard basis vector as a column. Hence, every coordinate of $u^\top$ is equal to $\pm 1$. Consider the $2\times n_1$ TU-matrix $\begin{psmallmatrix} A \\ u^\top \end{psmallmatrix}$. If there are $i \not=j$ such that $A_i = A_j \not=0$, then because of the TU-property also $(u^\top)_i = (u^\top)_j$ holds, hence, the $i$th and the $j$th column of $M'$ coincide, a contradiction.  Therefore, if there are $i \not=j$ such that $A_i \not=0 \not= A_j$, then $A_i = - A_j$, and by the same argument we also have $(u^\top)_i = -(u^\top)_j$, and thus, also the $i$th and the $j$th column of $M'$ are centrally-symmetric, a contradiction to $M'$ having column sums $1$. Also, we see in the same way that that there cannot exist $i \not=j$ such that $A_i = 0=A_j$. This shows that $A$ has only at most two entries, i.e., $n_1 \le 2$. Hence, the size $m_1+n_1$ of $A$ is at most $3$, a contradiction to the assumption of \Cref{4-sum}. This proves $m_1 \ge 2$. In the same way, one shows $m_2 \ge 2$. This finishes the proof of \Cref{theorem:tu}.
\hfill$\qed$

\medskip
\bibliographystyle{alpha}
\bibliography{tu-paper}

\end{document}